\documentclass[12pt,a4paper]{article}
  
\usepackage{amsmath,amssymb} 
\usepackage{graphicx} 
\usepackage{color} 
\usepackage{url} 
  
\newcommand{\mbf}[1]{\textbf{\textit{#1}}}
\newcommand{\mbb}{\mathbb}

\newcommand{\ov}{\overline} 
\newcommand{\un}{\underline} 
\newcommand{\m}{\mathrm{mid}\;} 
 
\newcommand{\Tol}{\mathrm{Tol}\,} 
 
\newcommand{\Ab}{(\mbf{A},\mbf{b})} 
\newcommand{\cond}{\mathrm{cond}\,} 
  
\renewcommand{\r}{\mathrm{rad}\;} 
\definecolor{MyGreen}{rgb}{0.1,0.7,0.2} 
\definecolor{Gray1}{rgb}{0.6,0.6,0.65}
\definecolor{Gray2}{rgb}{0.4,0.45,0.4}
\definecolor{Gray3}{rgb}{0.6,0.55,0.55} 
  
\newcounter{defi}
\newcounter{theo}
\newcounter{exmp}
\newenvironment{definition}{\addvspace{\bigskipamount}
  \noindent{\bf Definition \arabic{defi}.}\sl}{\addtocounter{defi}{1}
  \par\addvspace{\bigskipamount}}
\newenvironment{theorem}{\addvspace{\bigskipamount}
  \noindent{\bf Theorem \arabic{theo}.}\sl}{\addtocounter{theo}{1}
  \par\addvspace{\bigskipamount}}
\newenvironment{proof}{\noindent{\sl Proof.}}{\qquad$\Box$
  \par\addvspace{\bigskipamount}} 
\newenvironment{example}{\addvspace{3ex}\noindent\addtocounter{exmp}{1}%
  {\textbf{Example~\arabic{exmp}\;}}}{\hfill$\blacksquare$\par\addvspace{3ex}}
  
\title{\bf Interval regularization \\[4pt]  
       for imprecise linear algebraic equations} 
\author{Sergey P. Shary}
\date{\small Institute of Computational Technologies SB RAS, \\ 
          6 Lavrentiev ave., 630090 Novosibirsk, Russia,     \\ 
                    \texttt{shary@ict.nsc.ru}} 
  
\begin{document}
\maketitle
  
\begin{abstract}
In this paper, we consider the solution of ill-conditioned systems of linear algebraic 
equations that can be determined imprecisely. To improve the stability of the solution 
process, we ``immerse'' the original imprecise linear system in an interval system of 
linear algebraic equations of the same structure and then consider its \emph{tolerable 
solution set}. As the result, the ``intervalized'' matrix of the system acquires close 
and better conditioned matrices for which the solution of the corresponding equation 
system is more stable. \\ 
As a pseudo-solution of the original linear equation system, we take a point from 
the tolerable solution set of the intervalized linear system or a point that provides 
the largest tolerable compatibility (consistency). We propose several computational 
recipes to find such pseudo-solutions. \\[2mm]
\end{abstract} 
  
\bigskip 
  
\section{Problem statement}

In our work, we consider using methods of interval analysis for the solution 
of ill-conditioned systems of linear algebraic equations that can be specified 
imprecisely or inaccurately. We are developing a procedure for regularization of 
such problems, i.\,e., for improving stability of the process of solving them, which 
is called ``interval regularization''. 
  
Let us be given a system of linear algebraic equations of the form 
\begin{equation}
\label{LinSystem1} 
\arraycolsep 1pt
\left\{ \
\begin{array}{ccccccccc}
a_{11} x_1 &+& a_{12} x_2 &+&\ldots&+& a_{1n} x_n &=& b_1, \\[3pt]
a_{21} x_1 &+& a_{22} x_2 &+&\ldots&+& a_{2n} x_n &=& b_2, \\[3pt]
\vdots && \vdots && \ddots &&\vdots &&\vdots \\[3pt]
a_{m1} x_1 &+& a_{m2} x_2 &+&\ldots&+& a_{mn} x_n &=& b_m,
\end{array}
\right.
\end{equation}
with coefficients $a_{ij}$ and right-hand sides $b_i$, or, in concise form, 
\begin{equation}
\label{LinSystem2}
A x = b
\end{equation}
where $A = (\,a_{ij})$ is an $m\times n$-matrix and $b = (\,b_i)$ is a right-hand side 
vector. In our paper, we mainly consider the square case $m = n$, but some of our 
constructions are more general and they can be applied to rectangular linear systems 
with $m\neq n$. 
  
In the linear system \eqref{LinSystem1}--\eqref{LinSystem2}, the matrix $A$ may be 
ill-conditioned or even singular. The system may have no solutions at all in the 
classical sense. Also, it can be specified imprecisely, with some measure of imprecision 
given. Our task is to find a solution  or a pseudo-solution (its substitute defined 
in a reasonable sense) for the system of equations \eqref{LinSystem1}--\eqref{LinSystem2} 
in a stable way. 
  
Since we are going to use methods of interval analysis in our work, the imprecision 
in specifying the systems of linear algebraic equations will be described using 
the interval concepts too. In accordance with the informal international standard 
\cite{INotation} which is used throughout this work, we designate intervals and 
interval values in bold, while usual non-interval (point) objects are not marked in 
any way. So, instead of the system of equations \eqref{LinSystem1}--\eqref{LinSystem2}, 
we shall have an interval system of linear algebraic equations 
\begin{equation}
\label{InLSystem1} 
\arraycolsep 1pt
\left\{ \
\begin{array}{ccccccccc}
\mbf{a}_{11} x_1 &+& \mbf{a}_{12} x_2 &+&\ldots&+& \mbf{a}_{1n} x_n &=& \mbf{b}_1, \\[3pt]
\mbf{a}_{21} x_1 &+& \mbf{a}_{22} x_2 &+&\ldots&+& \mbf{a}_{2n} x_n &=& \mbf{b}_2, \\[3pt]
\vdots && \vdots && \ddots &&\vdots &&\vdots \\[3pt]
\mbf{a}_{m1} x_1 &+& \mbf{a}_{m2} x_2 &+&\ldots&+& \mbf{a}_{mn} x_n &=& \mbf{b}_m, 
\end{array}
\right.
\end{equation}
with interval coefficients $\mbf{a}_{ij}$ and interval right-hand sides $\mbf{b}_i$, or, 
in concise form, 
\begin{equation}
\label{InLSystem2}
\mbf{A}x = \mbf{b}, 
\end{equation} 
where $\mbf{A} = (\,\mbf{a}_{ij})$ is an interval matrix and $\mbf{b} = (\,\mbf{b}_i)$ 
is an interval right-hand side vector. The major part of our constructions below is 
insensitive to such a change in the object under study. The interval linear system 
\eqref{InLSystem1}--\eqref{InLSystem2} is then considered as a family of point linear 
systems of the form \eqref{LinSystem1}--\eqref{LinSystem2} which are equivalent 
to each other to within a prescribed accuracy specified by the intervals 
in $\mbf{A}$ and $\mbf{b}$. 
  
  
\section{Idea of the solution}

We are going to rely on the following fact from matrix theory. Let $A$ be 
an $n\times n$-matrix and its condition number $\mathrm{cond}(A) =\|A\|\cdot\|A^{-1}\|$, 
defined for a subordinate norm $\|\cdot\|$, satisfies $\mathrm{cond}(A) > 1$. Then, in 
any neighbourhood of the matrix $A$, there are matrices $\tilde{A}$ having better 
condition number $\cond(\tilde{A}) < \cond(A)$. This follows from that the condition 
number does not have local minima, except for the global one --- namely, 
$\mathrm{cond}(A) = 1$ for subordinate norms. 
  
As a result, one naturally arrives at the following idea: we can replace the solution 
of the original system $Ax = b$ by the solution of the system $\tilde{A}x = b$ 
with close, but better conditioned matrix $\tilde{A}$. Under favorable circumstances, 
the solution to the new system will be close to the desired solution of the original 
system \eqref{LinSystem1}--\eqref{LinSystem2}. 
  
The idea we have just formulated is not new. There exists the \emph{Lavrentiev 
regularization method} \cite{Lavrentiev,LavrentievSaveliev} (see also \cite{Kabanikhin, 
TikhonovArsenin}), a popular regularization technique for the integral equations of 
the first kind and similar operator equations, and its essense is almost the same as 
the above stated idea. Imposing a small perturbation on the operator involved in the 
equation, we shift its small eigenvalues from zero and, hence, the operator moves away 
from singularity. This improves stability of the solution. 
  
The Lavrentiev regularization method also applies to systems of linear algebraic equations 
of the form \eqref{LinSystem1}--\eqref{LinSystem2}. In the simplest case, when the matrix 
$A$ is, e.\,g., symmetric and positive semidefinite, we should solve 
\begin{equation*} 
(A + \theta I)\,x = b 
\end{equation*} 
instead of the equation system \eqref{LinSystem1}--\eqref{LinSystem2}, where $I$ is 
the identity matrix and real number $\theta > 0$ is a~shift parameter. If $\lambda(A)$ are 
eigenvalues of $A$, then the eigenvalues of $A + \theta I$ becomes $\lambda(A) + \theta$, 
and the condition number with respect to the spectral norm is 
\begin{equation*} 
\cond(A + \theta I)\; = \;\frac{\lambda_{\max}(A) + \theta}{\lambda_{\min}(A) + \theta}. 
\end{equation*} 
It obviously decreases in comparison with $\cond(A) = \lambda_{\max}(A)/\lambda_{\min}(A)$
since the function 
\begin{equation*} 
f(x) = \frac{b + x}{a + x} = 1 + \frac{b - a}{a + x} 
\end{equation*} 
is evidently decreasing for $x > 0$ under $b > a\geq 0$. 
  
The Lavrentiev regularization is widely used for various equations and systems 
of equations, when the properties of $A$ are \emph{a priori} known, and the most 
important of them is information on how the spectrum of $A$ is located. In general, 
when we know nothing about the properties of the matrix $A$, the choice of the parameter 
$\theta$, i.\,e., the direction of the shift and its magnitude, is not evident. 
  
Turning to our idea, the main question is how to choose a better conditioned matrix 
$\tilde{A}$ near $A$? In other words, where and how to move the matrix $A$, if we do not 
know its properties? 
  
The unexpected implementation of our idea in the case when no information about $A$ 
is available may be to perform a shift of $A$ ``in all directions'' at the same time. 
Then there is certaintly a suitable direction among our shifts, ant it will provide 
desirable regularization and improvement of the matrix. 
  
  
\begin{figure}[htb]
\centering 
\unitlength=1.1mm 
\begin{picture}(100,50)
\put(0,14){\includegraphics[width=44mm]{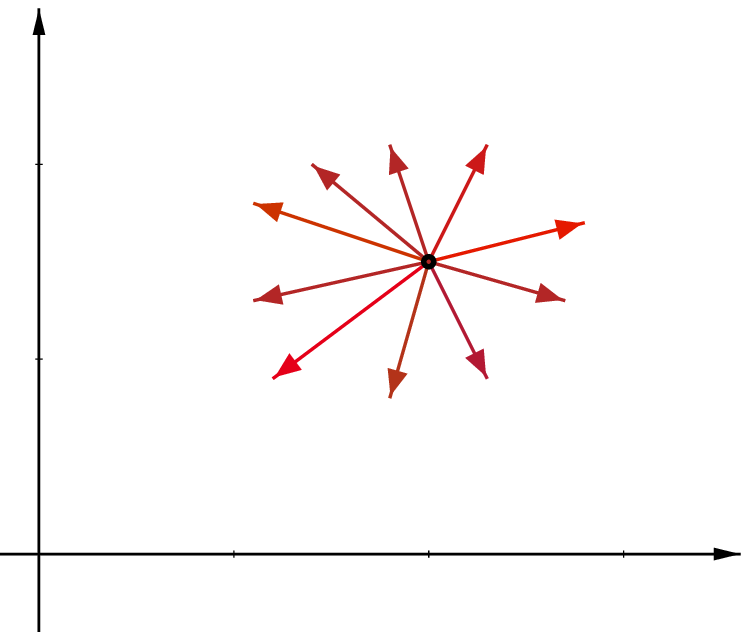}} 
\put(45,25){\color{Gray2}\rotatebox{-20}{\scalebox{1.8}{$\dashrightarrow$}}} 
\put(60,2){\includegraphics[width=44mm]{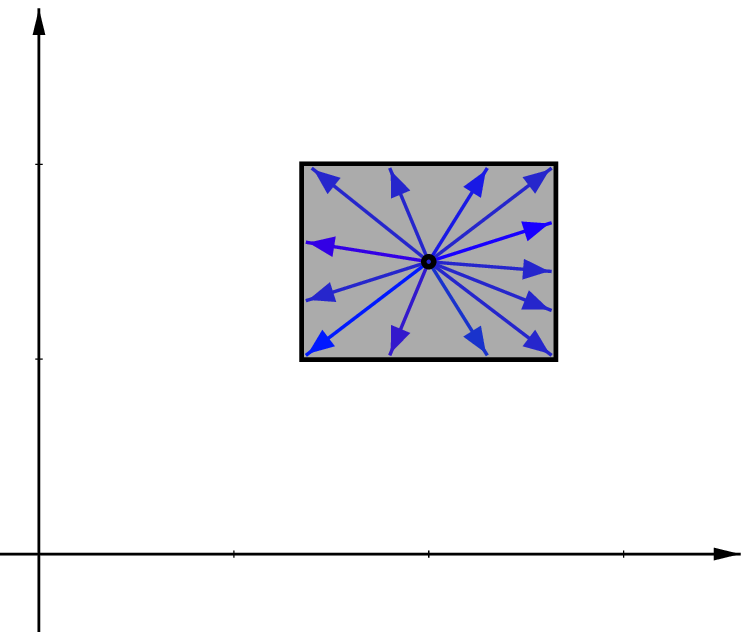}}
\end{picture} 
\caption{Displacement in all directions and all distances simultaneously}
is equivalent to covering a~neighborhood of the initial point. 
\label{ShiftPic} 
\end{figure}
  
  
Within the traditional data types used in calculus and numerical analysis, it is 
hardly possible to put into practice such an exotic recipe, but relevant tools have 
been already created in interval analysis (see, for example, \cite{IntervalComputs, 
IntervalAnalysis,GMayer,MooreBakerCloud,Neumaier,SharyBook}). With their help, 
our idea gets an elegant embodiment. 
  
In order to reach, with guarantee, the matrix $\tilde{A}$ no matter where it is, 
we shift the original matrix $A$ in all directions and to all possible distances 
that do not exceed a predetermined value $\theta$ (see Fig.~\ref{ShiftPic}) in 
a specified norm. This is equivalent to enclosing an entire neighborhood 
of the matrix~$A$. 
  
In interval terms, we ``inflate'' the matrix $A$, thus turning it into an interval 
matrix~$\mbf{A}$. To cover all possible shift directions of the matrix $A$, we assign 
\begin{equation*} 
\mbf{A} = A + \theta\mbf{E}, 
\end{equation*} 
where $\mbf{E} = ([-1, 1])$ is the matrix, of the same size as $A$, made up of the 
intervals $[-1, 1]$, and $\theta$ is the parameter of the ``inflation'' value. 
In general, instead of the equation system \eqref{LinSystem1}--\eqref{LinSystem2}, 
we come to the need to ``solve'' the interval system of linear algebraic equations 
\begin{equation} 
\label{InLSystem3}
\mbf{A}x = b, 
\end{equation} 
having the form \eqref{InLSystem1}--\eqref{InLSystem2}, and the solution process 
must be stable. In particular, it is desirable to base the solution process 
on well-conditined matrices within $\mbf{A}$. 
  
Notice that our construction is more general and, possibly, more flexible than 
the Lavrentiev regularization, since we use the matrix $A + \theta\mbf{E}$ instead 
of just $A + \theta I$, that is, we can perturbate off-diagonal elements of $A$ too. 
  
\begin{example} 
\label{InflaCondExmp} 
As an example demonstrating the evolution of the condition number after a point matrix 
inflates to an interval one, we consider the matrix 
\begin{equation*} 
\arraycolsep=3pt 
A = \left( 
\begin{array}{cc}
 99 & 100 \\[3pt] 
 98 &  99 
\end{array}
\right). 
\end{equation*} 
With respect to the spectral matrix norm $\|A\| = \sqrt{\lambda_{\max}(A^{\top}\!A)}$, 
the condition number of the matrix is $\cond(A) = 3.92\cdot 10^4$, and it is not hard 
to show that this is the maximum for regular $2\times 2$-matrices with positive integer 
elements $\leq 100$. 
  
Let us ``intervalize'' the matrix $A$ by adding $[-1, 1]$ to each element. We get 
\begin{equation*} 
\arraycolsep=3pt 
\mbf{A} = \left( 
\begin{array}{cc}
{[98, 100]} & {[99, 101]} \\[3pt] 
{[97, 99]} &  {[98, 100]} 
\end{array}
\right). 
\end{equation*} 
  
The new interval matrix acquires a singular point matrix 
\begin{equation*} 
\arraycolsep=3pt 
\left( 
\begin{array}{cc}
 98 & 99 \\[3pt] 
 98 & 99 
\end{array}
\right) 
\end{equation*} 
and many more singular matrices. The condition numbers of the ``endpoint matrices'' 
of the intervalized matrix $\mbf{A}$ are equal to 
  
\begin{center} 
\tabcolsep=7pt 
\begin{tabular}{cccc}  
$3.84\cdot 10^4$, &     $197.02$,     &     $201.12$,     & $1.31\cdot 10^4$, \\[3pt] 
   $197.02$,      &     $98.76$,      & $1.31\cdot 10^4$, &    $195.12$,      \\[3pt] 
   $197.0$,       & $3.92\cdot 10^4$, &      $99.26$,     &    $199.02$       \\[3pt] 
$3.92\cdot 10^4$, &     $199.00$,      &     $199.02$,     &  $4.0\cdot 10^4$.
\end{tabular} 
\end{center} 
  
\smallskip\noindent 
We can see that, among 16 endpoint matrices, one matrix has larger condition number 
$4.0\cdot 10^4$, two matrices have the same condition number, and one matrix is slightly 
better conditined. However, 10 matrices of 16 have considerably smaller condition numbers 
$\leq 200$. A more thorough numerical test shows that the condition number $98.76$, 
attained at the endpoint matrix 
\begin{equation*} 
\arraycolsep=3pt 
\left( 
\begin{array}{cc}
100 &  99 \\[3pt] 
 97 & 100 
\end{array}
\right), 
\end{equation*} 
is really minimal among all the condition numbers of the point matrices from $\mbf{A}$. 
We will further discuss this phenomenon in Section~\ref{TolSolSetSect}. 
\end{example} 
  
Another observation is that not only well-conditioned point matrices fall into 
the interval matrix $\mbf{A}$ after intervalization of $A$. Ill-conditioned and even 
singular matrices also appear in $\mbf{A}$. Our task is to construct the solution 
process in such a way that it relies mainly on well-conditioned matrices from $\mbf{A}$. 
  
  
\section{Implementation of the idea} 
\label{ImplemIdeaSect}

In modern interval analysis, the concept of ``solution'' of an interval equation or 
a system of equations can be understood in various ways which are very different from 
each other. As a rule, the solutions to interval problems are estimates (most often, 
also interval ones) of some ``solution sets'' arising in connection with the interval 
problem statement. In its turn, the ``solution sets'' are usually determined from 
solutions to separate point problems forming the interval problem under study, but 
that can be done in various ways depending on the types of uncertainty that the input 
data intervals express. 
  
The fact is, the interval data uncertainty has, in its essense, a dualistic and 
ambivalent character \cite{SharySurvey,SharyBook}. In the formal setting of any 
interval problem, we need to distinguish between the so-called uncertainties of 
the A-type and E-type, or, briefly, \emph{A-uncertainty} and \emph{E-uncertainty}: 
\begin{itemize} 
\item 
the uncertainty of the A-type (A-uncertainty) corresponds to the application \\ 
of the logical quantifier ``$\forall$'' to the interval variable, that is, when \\
the condition ``$\forall x\in\mbf{x}$'' enters the definition of the solution set; 
\item 
the uncertainty of the E-type (E-uncertainty) corresponds to the application \\ 
of the logical quantifier ``$\exists$'' to the interval variable, that is, when \\  
the condition ``$\exists x\in\mbf{x}$'' enters the definition of the solution set. 
\end{itemize} 
Sometimes, in connection with the properties expressed by interval A-uncertainties 
and E-uncertainties, the terms a \emph{strong property} and a \emph{weak property} are 
used.
  
As a consequence, different solution sets for interval systems of equations and other 
interval problems can be defined by various combinations of these quantifiers applied 
to interval parameters. The simplest and most popular among the solution sets is the 
set obtained by collecting all possible solutions of non-interval (point) equations or 
systems of equations which we get by fixing the parameters of the system within 
specified intervals. This is the ``united solution set''. 
  
\begin{definition} 
For the interval system of linear algebraic equations \eqref{InLSystem1}%
--\eqref{InLSystem2}, the set 
\begin{align*} 
\Xi_{\mathit{uni}}\Ab \  &\stackrel{\text{def}}{=} \   \bigl\{\, 
  x\in\mbb{R}^n \mid (\exists A\in\mbf{A})(\exists b\in\mbf{b})(Ax = b)\,\bigr\} \\[2pt] 
&= \,\ \bigl\{\,x\in\mbb{R}^n \mid (\exists A\in\mbf{A})(Ax\in\mbf{b})\,\bigr\}. 
\end{align*} 
is called \textit{united solution set}. 
\end{definition} 
  
The above definition is organized according to the separation axiom from the formal 
set theory (which is also known as ``axiom schema of specification'' or ``subset axiom 
scheme''): ``Whenever the propositional function $P(x)$ is definite for all elements 
of a set $M$, there exists a subset $M'$ in $M$ that contains precisely those elements 
$x$ of $M$ for which $P(x)$ is true'' (see, e.\,g., \cite{Halmos}). The united solution 
set corresponds to the situation when $M = \mbb{R}^n$, $P(x)$ is a predicate with 
the existential quantifiers ``$\exists$'' applied to all interval parameters of 
the system of equations. The equivalent set-theoretical representation of the united 
solution set is 
\begin{eqnarray} 
\Xi_{\mathit{uni}}\Ab \  & = & \  \bigcup_{A\in\mbf{A}}   \nonumber \ 
  \bigcup_{b\in\mbf{b}} \  \bigl\{\,x\in\mbb{R}^n \mid Ax = b\,\bigr\}  \\[3mm] 
  & = & \      \label{UniSolSet} 
  \bigcup_{A\in\mbf{A}} \  \bigl\{\,x\in\mbb{R}^n \mid Ax\in\mbf{b}\,\bigr\}, 
\end{eqnarray} 
where $\{\,x\in\mbb{R}^n \mid Ax\in\mbf{b}\,\}$ is, in fact, the solution set to 
the ``partial'' equation system $Ax = \mbf{b}$. 
  
The united solution set carefully takes into account the contributions of all point 
equation systems forming an interval system, by means of uniting their separate 
solutions together. Accordingly, the united solution set is subject to variability 
in the same extent as this variability is inherent to solutions of the individual 
point systems from the interval system under study. If the interval system of linear 
equations includes ill-conditioned or singular point systems, for which the solution 
varies greatly as the result of data perturbations, then the united solution set 
includes all these variations and will not play any stabilizing role. This will 
inevitably happen after intervalization of system \eqref{LinSystem1}--\eqref{LinSystem2} 
in case it is ill-conditioned. 
  
Overall, the united solution set is not really suitable for a stable solution 
of the system $Ax = b$: its stability is determined by solutions of the most unstable 
systems due to representation \eqref{UniSolSet}. Working with the united solution set 
requires an additional regularization procedure, e.\,g., such as that proposed 
by A.N.\,Tikhonov in \cite{Tikhonov}. We are going to develop another approach 
that relies on good properties of a specially selected solution set. 
  
First of all, we require that the solution set of an interval system should be 
constructed from the most stable solutions of point systems forming the interval 
system of equations. What is this solution set?\ldots We will not intrigue the reader 
and immediately announce the answer: among the solution sets for interval systems of 
equations, the ``most stable'' and, as a consequence, the most suitable for 
regularization purposes is the so-called tolerable solution set. 
  
\begin{definition} 
For the interval linear algebraic system \eqref{InLSystem1}--\eqref{InLSystem2}, 
the set 
\begin{equation} 
\label{TolSolSet}
\Xi_{\mathit{tol}}\Ab \  \stackrel{\text{def}}{=} \ 
  \bigl\{\, x\in\mbb{R}^n \mid(\forall A\in\mbf{A})(\exists b\in\mbf{b})(Ax = b)\,\bigr\}, 
\end{equation} 
is called \textit{tolerable solution set}. 
\end{definition} 
  
The tolerable solution set is composed of all such vectors $x\in\mbb{R}^n$ that 
the product $Ax$ falls into the interval of the right-hand side $\mbf{b}$ for any 
matrix $A\in\mbf{A}$. The definition \eqref{TolSolSet} can also be rewritten in 
the equivalent form 
\begin{equation*} 
\Xi_{\mathit{tol}}\Ab 
  = \bigl\{\, x\in\mbb{R}^n \mid (\forall A\in\mbf{A})(Ax\in\mbf{b})\,\bigr\}. 
\end{equation*} 
  
The presence of the condition ``$\forall A\in\mbf{A}$'' with the universal 
quantifier in the definition of the tolerable solution set results in the fact that 
the set-theoretic representation of $\Xi_{\mathit{tol}}\Ab$ uses the intersection 
over $A\in\mbf{A}$ rather than the union, as was the case with $\Xi_{uni}\Ab$. 
Therefore, instead of \eqref{UniSolSet}, we get 
\begin{eqnarray} 
\Xi_{\mathit{uni}}\Ab \  & = & \  \bigcap_{A\in\mbf{A}}   \nonumber \ 
  \bigcup_{b\in\mbf{b}} \  \bigl\{\,x\in\mbb{R}^n \mid Ax = b\,\bigr\} \\[3mm] 
  & = & \      \label{SetTheoRepr} 
  \bigcap_{A\in\mbf{A}} \  \bigl\{\,x\in\mbb{R}^n \mid Ax\in\mbf{b}\,\bigr\}. 
\end{eqnarray} 
  
The representation \eqref{SetTheoRepr} shows that the tolerable solution set is 
the least sensitive to changes in the matrix among all the solution sets of interval 
linear systems, since it is not greater than the ``most stable'' solution sets 
$\bigl\{\, x \in\mbb{R}^n \mid \tilde{A}x \in\mbf{b}\,\bigr\}$ determined by the matrix 
$\tilde{A}$ with the best condition number from $\mbf{A}$. Although some point matrices 
from $\mbf{A}$ may be poorly conditioned or even singular, their effect is compensated 
by the presence, in the same interval matrix, of ``good'' point matrices that make 
the tolerable solution set bounded and stable as a whole. 
  
The principal difference between the tolerable solution set and united solution set is 
expressed, in particular, in the fact that when the interval matrix $\mbf{A}$ widens, 
the united solution set of the system $\mbf{A}x = \mbf{b}$ expands too, while 
the tolerable solution set shrinks, i.\,e., decreases in size. 
  
To sum up, for the interval system of linear algebraic equations obtained after 
``intervalization'' of the initial ill-conditioned system, we shall consider 
the tolerable solution set $\Xi_{\mathit{tol}}\Ab$. We are interested in the points 
from it or its estimates. The problem of studying and estimating the tolerable solution 
set for interval linear systems of equations is called the interval linear tolerance 
problem \cite{SharyMCS,SharyARC,SharyBook}. We, therefore, need its solution, perhaps 
a partial one, which will be taken as a \emph{pseudo-solution} to the original equation 
system \eqref{LinSystem1}--\eqref{LinSystem2} or \eqref{InLSystem1}--\eqref{InLSystem2} 
instead of the ideal solution that may be unstable or even non-existing. 
  
At this point, we are confronted with a specific feature of the tolerable solution 
set to interval systems of equations: it is often empty, which can happen even for 
ordinary data. For system \eqref{InLSystem1}--\eqref{InLSystem2}, this is the case 
when the intervals of the right-hand sides $\mbf{b}_i$ are ``relatively narrow'' 
in comparison with intervals in the matrix $\mbf{A}$. Then the range of all possible 
products of $Ax$ for $A\in\mbf{A}$ exceeds the width of the ``corridor'' of the 
right-hand side $\mbf{b}$ into which this product should fit. 
  
For example, the tolerable solution set is empty for the one-dimensional interval 
equation $[1, 2]\,x = [3, 4]$. On the one hand, zero cannot be in the tolerable solution 
set, since $[3, 4]\not\ni 0$. On the other hand, a non-zero real number $t$ cannot be 
in the tolerable solution set too, since the numbers from the range of $[1, 2]\,t$ 
can differ by a factor of two, whereas the right-hand side $[3, 4]$ can take only 
the difference of numbers by a factor of $4/3$. 
  
In order to make the tolerable solution set non-empty, we can artificially widen 
the right-hand side of the interval linear equation system, for example, uniformly 
with respect to the midpoints of the interval components. It is not difficult to realize 
that, with the help of such an expansion, we can always make the tolerable solution set 
non-empty. 
  
An alternative way is to consider not the tolerable solution set itself, but 
a quantitative measure of the solvability of the linear tolerance problem, and 
the points at which the maximum of this measure is reached will be declared 
pseudo-solutions. This approach is developed in Section~\ref{RecFuncSect} 
of the present work.

  
\section{Tolerable solution set \\*  for interval linear systems of equations} 
\label{TolSolSetSect}

The tolerable solution set was first considered in \cite{NudingWilhelm} under the name 
of \emph{restricted solution set}, which is possibly due to the fact that this set is 
usually much smaller than the common and well-studied united solution set. Both the united 
and tolerable solution sets are representatives of an extensive class of the so-called 
\emph{AE-solution sets} for interval systems of equations (see \cite{SharySurvey,SharyBook}). 
It is not difficult to show that the AE-solution sets are polyhedral sets, i.\,e., their 
boundaries are made up of pieces of hyperplanes. But the tolerable solution set for interval 
linear systems has even better properties: it is a convex polyhedral set in $\mbb{R}^n$ 
(see  \cite{Irene05,SharyMCS,SharyBook}), i.\,e., it can be represented as the intersection 
of finite number of closed half-spaces of $\mbb{R}^n$. 
  
\begin{example} 
Let us consider the interval linear system 
\begin{equation} 
\label{ExampleInSys} 
\arraycolsep=3pt 
\left( 
\begin{array}{ccc}
  2.8    & {[0, 2]} & {[0, 2]} \\[2pt] 
{[0, 2]} &   2.8    & {[0, 2]} \\[2pt] 
{[0, 2]} & {[0, 2]} &   2.8 
\end{array} 
\right) 
x = 
\left( 
\begin{array}{c}
{[-1, 1]} \\[2pt] 
{[-1, 1]} \\[2pt] 
{[-1, 1]} 
\end{array} 
\right), 
\end{equation} 
proposed in \cite{Reichmann} and later studied in \cite{Neumaier}. 
  
  
\begin{figure}[htb] 
\centering
\unitlength=1mm 
\begin{picture}(90,82) 
\put(0,2){\includegraphics[width=90mm]{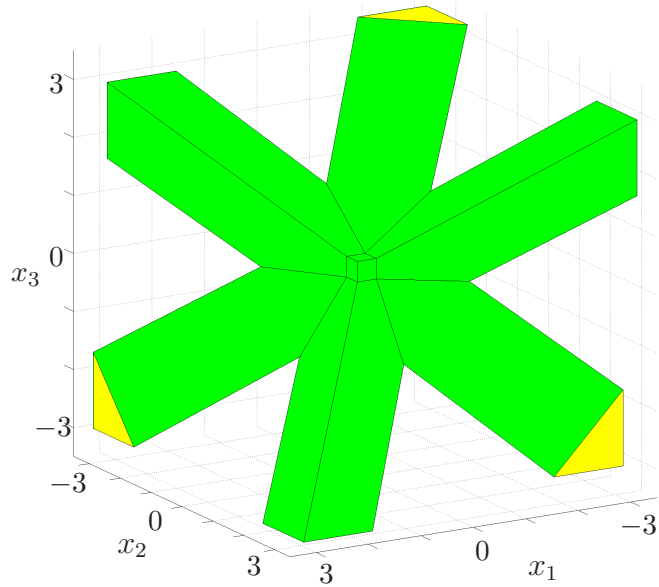}} 
\put(68,3){$x_1$} 
\put(14,6){$x_2$} 
\put(80,8.5){\small$-3$} 
\put(61,5){\small$0$} 
\put(40.5,2){\small$3$} 
\put(5,14){\small$-3$} 
\put(18,9){\small$0$} 
\put(31,3){\small$3$} 
\put(0,42){\small$x_3$} 
\put(3,21.5){\small$-3$} 
\put(5,44){\small$0$} 
\put(5,67){\small$3$} 
\end{picture} 
\caption{Unbounded united solution set to the interval system \eqref{ExampleInSys}.} 
\label{UniSolSetPic} 
\end{figure} 
  
  
For the value of the diagonal elements $3.5$ in the matrix of \eqref{ExampleInSys}, 
its united solution set is depicted at the jacket of the book \cite{Neumaier}. In our 
specific case, when the diagonal elements are equal to $2.8$, the interval matrix of
\eqref{ExampleInSys} contains singular point matrices, and the united solution set 
becomes unbounded. 
  
  
\begin{figure}[htb]
\centering
\unitlength=1mm
\begin{picture}(70,73)
\put(0,1){\includegraphics[width=70mm]{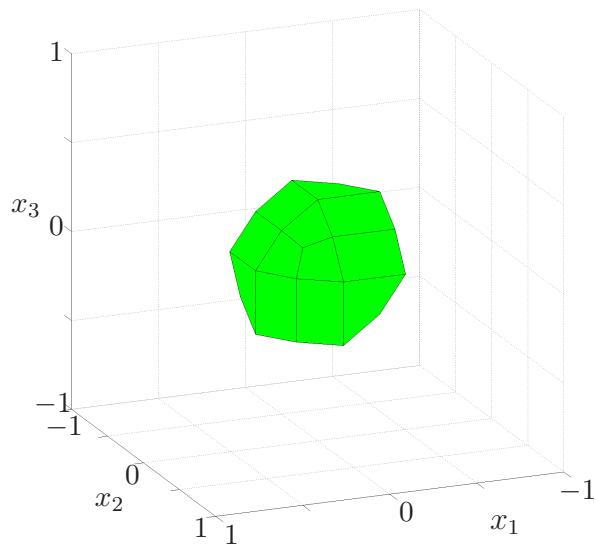}}
\put(58,3){$x_1$} 
\put(67,7){\small$-1$} 
\put(46,4){\small$0$} 
\put(23,1){\small$1$} 
\put(6,6){\small$x_2$} 
\put(-0.5,15.5){\small$-1$} 
\put(10,9){\small$0$} 
\put(19,2){\small$1$} 
\put(-5,45){\small$x_3$} 
\put(-2,18.5){\small$-1$} 
\put(0,42){\small$0$} 
\put(0,65){\small$1$} 
\end{picture}
\caption{Tolerable solution set to the interval system \eqref{ExampleInSys}.} 
\label{TolSolSetPic} 
\end{figure}
  
  
Nevertheless, both united solution set and tolerable solution set for the interval 
system \eqref{ExampleInSys} can be visualized with the use of the free software 
package \texttt{IntLinInc3D} \cite{IreneSoftware}, and their pictures are presented 
at Fig.~\ref{UniSolSetPic} and Fig.~\ref{TolSolSetPic}. The unbounded united solution 
set infinitely extends beyond the boundaries of the picture box through the light 
trimming faces at Fig.~\ref{UniSolSetPic}. However, the tolerable solution set to 
the system \eqref{ExampleInSys} is bounded and quite small (see Fig.~\ref{TolSolSetPic}). 
The reduction of the solution set, the pruning of its infinite parts, illustrates how 
efficiently the transition to the tolerable solution set ``regularizes'' the singular 
interval system \eqref{ExampleInSys}. 
\end{example} 
  
There exists several results that provide us with analytical descriptions of the tolerable 
solution sets to interval linear systems of equations. 
  
\begin{theorem} {\rm(the Rohn theorem \cite{Rohn85,RohnHandbook,SharyBook})} 
A point $x\in\mbb{R}^n$ belongs to the tolerable solution set of the interval 
$m\times n$-system of linear algebraic equations $\mbf{A}x = \mbf{b}\,$ if and only 
if $\,x = x' - x''$ for some vectors $\,x'$, $x''\in\mbb{R}^{n}$ that satisfy 
the following system of linear inequalities 
\begin{equation}
\label{RohnSystem} 
\arraycolsep 3pt 
\left\{ \ 
\begin{array}{rcr} 
 \ov{\mbf{A}}x' - \un{\mbf{A}}x'' & \leq &  \ov{\mbf{b}},\\[3pt] 
-\un{\mbf{A}}x' + \ov{\mbf{A}}x'' & \leq & -\un{\mbf{b}},\\[3pt] 
x',\, x'' & \geq & 0, 
\end{array}
\right.
\end{equation} 
where $\un{\mbf{A}}$, $\ov{\mbf{A}}$, $\un{\mbf{b}}$, $\ov{\mbf{b}}$ denote lower 
and upper endpoint matrices and vectors for $\mbf{A}$ and $\mbf{b}$ respectively. 
\end{theorem} 
  
\begin{theorem} {\rm(Irene Sharaya's theorem \cite{Irene05})} 
Let $\mbf{A}_{i:}$ be the $i$-th row of the interval $m\times n$-matrix $\mbf{A}$, 
and $\mathrm{vert}\;\mbf{A}_{i:}$ denotes the set of vertices of this interval vector, 
i.\,e., the set $\bigl\{\,(\tilde{a}_{i1},\ldots,\tilde{a}_{in}) \mid \tilde{a}_{ij}
\in\{\un{\mbf{a}}_{ij}, \ov{\mbf{a}}_{ij}\}, j = 1,2,\ldots,n\,\}$. For an interval 
system of linear algebraic equations $\mbf{A}x = \mbf{b}$, the tolerable solution 
set $\Xi_{\mathit{tol}}\Ab$ can be represented in the form 
\begin{equation} 
\label{IreneRepresntn} 
\Xi_{\mathit{tol}}\Ab \ 
= \  \bigcap_{i=1}^m \  \bigcap_{a\in\mathrm{vert}\,\mbf{A}_{i:}} 
  \{\, x\in\mbb{R}^n \mid ax \in\mbf{b}_{i}\}, 
\end{equation} 
i.\,e., as the intersection of hyperstrips $\{\,x\in\mbb{R}^n \mid ax \in\mbf{b}_{i}\}$. 
If $|M|$ means cardinality of a finite set $M$, then the number of hyperstrips in 
the intersection \eqref{IreneRepresntn} does not exceed $\,\sum^m_{i=1} |\,\mathrm{vert}
\;\mbf{A}_{i:}|$ and, a fortiori, does not exceed $m\cdot 2^n$. 
\end{theorem} 
  
Each of the inclusions $ax\in\mbf{b}_{i}$ for $a\in\mbf{A}_{i:}$ is equivalent to 
a two-sided linear inequality 
\begin{equation*} 
\un{\mbf{b}}_{i}\leq a_{i1}x_1 + a_{i2}x_2 + \ldots + a_{in}x_n \leq\ov{\mbf{b}}_{i},  
\end{equation*} 
which really determines a hyperstrip in $\mbb{R}^n$, i.\,e., a set between two parallel 
hyperplanes. Therefore, Irene Sharaya's theorem gives a representation of the tolerable 
solution set as the set of solutions to a finite system of two-sided linear inequalities 
whose coefficients are endpoints of the interval elements from $\mbf{A}_{i:}$, $i = 1,2, 
\ldots,m$. The remarkable fact is that the number of inequalities implied by the 
representation \eqref{IreneRepresntn} is considerably less than the overall number of 
``endpoint inequalities'' of the interval linear system which is equal to $2^{m(n+1)}$. 
  
  
\begin{figure}[ht] 
\label{IreneTheoPic}
\centering 
\unitlength 1mm 
\begin{picture}(96,78) 
\put(0,0){\includegraphics[width=88mm]{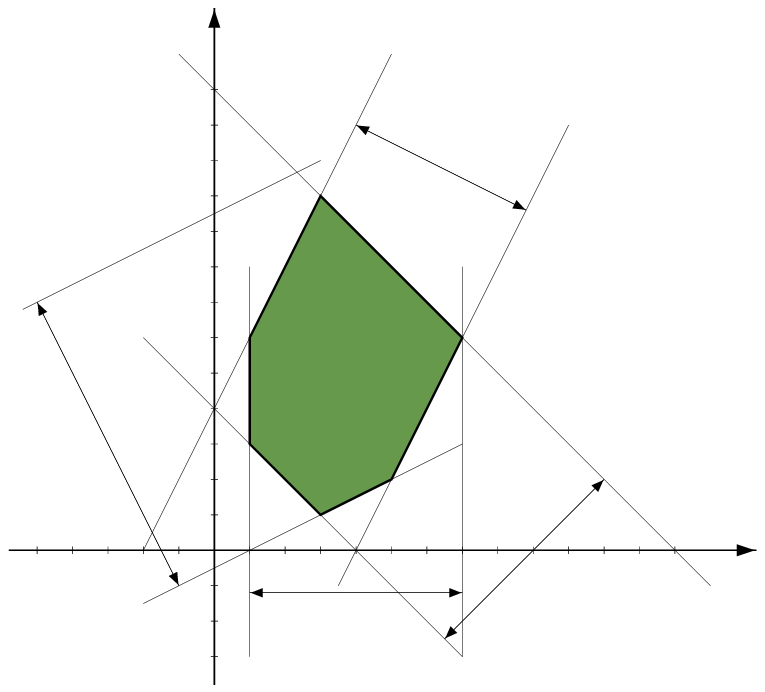}}
\put(23.3,29.5){\scriptsize 4}
\put(22.3,65){\scriptsize 13}
\put(16.3,15.3){\scriptsize -2}
\put(22.5,6){\scriptsize -2}
\put(41,15.3){\scriptsize 4}
\put(75.5,15.3){\scriptsize 13}
\put(48,59){\color{blue}|1|}
\put(59,10){\color{blue}|2|}
\put(38,6){\color{blue}|3|}
\put(6,24){\color{blue}|4|}
\put(68,44){{\color{blue}|$i$|}\, indicates a strip}
\put(68,39){\parbox[t]{60mm}{that corresponds to the solution\\ 
of the $i$-th row of the system \eqref{IreneSystem}}}
\put(82,11.7){$x_1$}
\put(27,72){$x_2$}
\put(37,32){\Large{$\Xi_\mathit{tol}$}} 
\end{picture} 
\caption{Constructing the tolerable solution set} 
according to Irene Sharaya's theorem. 
\end{figure} 
  
  
\begin{example} 
For the interval linear equation system 
\begin{equation} 
\label{IreneSystem}
\arraycolsep=2mm 
\left(\begin{array}{rr}
-2 & 1 \\[2pt] 1 & 1 \\[2pt] 1 & 0 \\[2pt] -1 & 2 
\end{array}\right)
\begin{pmatrix}x_1 \\[2pt] x_2\end{pmatrix} = 
\left(\begin{array}{@{\,}c@{\,}} 
{[-8, 4]}\\[2pt] {[4, 13]}\\[2pt] {[1, 7]}\\[2pt] {[-1, 19]} 
\end{array}\right), 
\end{equation}
the tolerable solution set can be constructed in the way depicted 
at Fig.~\ref{IreneTheoPic} which is borrowed from \cite{Irene05}. 
\end{example} 
  
From Irene Sharaya's theorem, it follows that, within the given interval matrix, 
the point matrix with the best condition number is necessarily an endpoint (``corner'') 
matrix. We saw an illustration of this fact in Example~\ref{InflaCondExmp} 
in Section \ref{ImplemIdeaSect}. 
  
As far as the solution of a system of linear inequalities is computable in polynomial 
time depending on the size of the problem (see, e.\,g., \cite{Khachiyan,Schrijver}), 
the Rohn theorem implies that, in general, the recognition of whether the tolerable 
solution set is empty or not empty is a polynomialy solvable problem too. 
  
Over the last decades, several approaches have been developed to study the tolerable 
solution set and to compute its estimates. These are: 
\begin{itemize} 
\item 
Application of systems of linear inequalities from theorems of Jiri Rohn and 
Irene Sharaya. 
\item 
Formal algebraic approach. 
Estimation of the tolerable solution set reduces to computing so-called formal 
(algebraic) solutions for a special interval linear system of the same form. 
\item 
The method of the recognizing functional. The tolerable solution set is represented 
as a level set of a special function called recognizing functional, and we study 
the problem by using the functional, its values and their sign. 
\end{itemize} 
  
Following the author's earlier ideas, a technique based on correction and further 
solution of the system of linear inequalities \eqref{RohnSystem} has been developed 
in \cite{PanyukovGolodov}. In our paper, we are going to elaborate the second and 
the third approaches which use purely interval technique and work directly with 
the interval system of equations. 
  
  
\section{Recognizing functional and its application} 
\label{RecFuncSect}

To go further and make our article self-sufficient, we need to recall some fundamental 
concepts and facts from interval analysis. 
  
The main instrument of interval analysis is so-called \emph{interval arithmetics}, 
algebraic systems that formalize common operations between entire intervals of the real 
line $\mbb{R}$ or other number fields. In particular, the classical interval arithmetic 
$\mbb{IR}$ is an algebraic system formed by intervals $\,\mbf{x}\; =\; [\,\un{\mbf{x}}, 
\ov{\mbf{x}}\,]\,\subset\,\mathbb{R}\,$ so that, for any arithmetic operation ``$\star$'' 
from the set $\{\,+\,,-\,,\,\cdot\,,/\,\}$, the result of the operation between 
the intervals is defined ``by representatives'', i.\,e., as 
\begin{equation*} 
\mbf{x}\star\mbf{y}\, = \,\bigl\{\; x\star y \ \mid \ x\in\mbf{x},\,y\in\mbf{y}\;\bigr\}. 
\end{equation*} 
The above formula is mainly of a theoretical nature, being hardly applicable for actual 
computations. Expanded constructive formulas for the interval arithmetic operations are 
as follows \cite{GMayer,MooreBakerCloud,Neumaier,SharyBook}: 
\begin{gather*}
\mbf{x} + \mbf{y} = 
  \bigl[\,\un{\mbf{x}}+\un{\mbf{y}}, \ \ov{\mbf{x}}+\ov{\mbf{y}}\,\bigr], \qquad 
\mbf{x} - \mbf{y} =  
  \bigl[\,\un{\mbf{x}}-\ov{\mbf{y}}, \ \ov{\mbf{x}}-\un{\mbf{y}}\,\bigr], \\[3pt]
\mbf{x}\cdot\mbf{y} =
  \bigl[\,\min\{\un{\mbf{x}}\,\un{\mbf{y}},\un{\mbf{x}}\,\ov{\mbf{y}},
  \ov{\mbf{x}}\,\un{\mbf{y}},\ov{\mbf{x}}\,\ov{\mbf{y}}\}, \ 
  \max\{\un{\mbf{x}}\,\un{\mbf{y}},\un{\mbf{x}}\,\ov{\mbf{y}},
  \ov{\mbf{x}}\,\un{\mbf{y}},\ov{\mbf{x}}\,\ov{\mbf{y}}\}\,\bigr],        \\[3pt]
\mbf{x}/\mbf{y} = \mbf{x}\cdot 
  \bigl[\,1/\ov{\mbf{y}}, \ 1/\un{\mbf{y}}\,\bigr]\qquad\mbox{ for }\ \mbf{y}
  \not\ni 0. 
\end{gather*} 
  
We start our consideration from the following characterization result for the points 
from the tolerable solution set (see \cite{SharyMCS,SharyARC,SharyBook}):  for 
an interval system of linear algebraic equations $\mbf{A}x = \mbf{b}\,$, the point 
$x\in\mbb{R}^n$ belongs to the solution set $\Xi_\mathit{tol}\Ab$  if and only if 
\begin{equation}
\label{TolCharact} 
\mbf{A}\cdot x \,\subseteq\, \mbf{b}, 
\end{equation} 
where ``$\,\cdot\,$'' means the interval matrix multiplication. The validity of this 
characterization follows from the properties of interval matrix-vector multiplication 
and the definition of the tolerable solution set. We are going to reformulate the 
inclusion \eqref{TolCharact} as an inequality, in order to be able to apply results 
of the traditional calculus. 
  
If $\mbf{A} = (\mbf{a}_{ij})$, then, instead of \eqref{TolCharact}, we can write 
\begin{equation*}
\sum_{j=1}^n \mbf{a}_{ij} x_{j} \subseteq \mbf{b}_{i}, \qquad i = 1,2,\ldots,m, 
\end{equation*} 
due to the definition of the interval matrix multiplication. Next, we represent 
the right-hand sides of the above inclusions as the sums of midpoints $\m\mbf{b}_{i}$ 
and intervals $\bigl[-\r\mbf{b}_{i}, \r\mbf{b}_{i}\bigr]$ which are symmetric with 
respect to zero (``balanced''): 
\begin{equation*} 
\sum_{j=1}^n \mbf{a}_{ij} x_{j} \;\subseteq\; \m\mbf{b}_{i} 
  + \bigl[-\r\mbf{b}_{i}, \r\mbf{b}_{i}\bigr],  \qquad i = 1,2,\ldots,m. 
\end{equation*}
Then, adding $(-\m\mbf{b}_{i})$ to both sides of the inclusions, we get 
\begin{equation*} 
\sum_{j=1}^n \mbf{a}_{ij} x_{j} \; - \m\mbf{b}_{i} \;\subseteq\; 
  \bigl[-\r\mbf{b}_{i}, \r\mbf{b}_{i}\bigr],  \qquad i = 1,2,\ldots,m. 
\end{equation*} 
  
The inclusion of an interval into the balanced interval $\bigl[-\r\mbf{b}_{i}, 
\r\mbf{b}_{i}\bigr]$ can be equivalently rewritten as the inequality on the absolute 
value: 
\begin{equation*} 
\left|\;\sum_{j=1}^n \mbf{a}_{ij} x_{j} \; - \m\mbf{b}_{i}\,\right|\; 
  \leq \;\r\mbf{b}_{i},  \qquad i = 1,2,\ldots,m, 
\end{equation*}
which implies 
\begin{equation*} 
\r\mbf{b}_{i} - \left|\;\sum_{j=1}^n \mbf{a}_{ij} x_{j} \; - \m\mbf{b}_{i}\,\right|\; 
  \geq \ 0,  \qquad i = 1,2,\ldots,m. 
\end{equation*} 
Therefore, 
\begin{equation*}
\mbf{A}x \subseteq\mbf{b} \quad\Longleftrightarrow\quad 
   \r\mbf{b}_{i} - \left|\,\m\mbf{b}_{i} - \sum_{j=1}^n \mbf{a}_{ij} x_{j}\,\right| \geq 0 
   \quad \text{ for each } \  i = 1,2,\ldots,m. 
\end{equation*}
Finally, we can convolve, over $i$, the conjunction of the inequalities in the right-hand 
side of the logical equivalence obtained: 
\begin{equation*}
\mbf{A}x \subseteq\mbf{b} \quad\Longleftrightarrow\quad 
\min_{1\leq i\leq m} \left\{ \  \r\mbf{b}_{i} - 
  \left|\,\m\mbf{b}_{i} - \sum_{j=1}^n \mbf{a}_{ij} x_{j}\,\right|\; \right\}\, \geq\, 0.
\end{equation*} 
  
We have arrived at the following result 
  
\begin{theorem} 
Let $\mbf{A}$ be an interval $m\times n$-matrix, $\mbf{b}$ be an interval $m$-vector. 
Then the expression 
\begin{equation*}
\Tol(x, \mbf{A}, \mbf{b})\,  = \,\min_{1\leq i\leq m}
  \left\{ \,\r\mbf{b}_i - \left|\; \m\mbf{b}_i - \sum_{j=1}^n 
  \,\mbf{a}_{ij} x_{j} \,\right| \,\right\} 
\end{equation*}
defines a mapping $\Tol : \mbb{R}^{n}\times\mbb{IR}^{m\times n}\times\mbb{IR}^{m} 
\rightarrow\mathbb{R}$, such that the memebership of a~point $\,x\in\mbb{R}^n$ in the 
tolerable solution set $\,\Xi_\mathit{tol}\Ab$ of an interval system of linear 
algebraic equation $\mbf{A}x = \mbf{b}$ is equivalent to that the mapping Tol is 
nonnegative in the point $x$, i.\,e. 
\begin{equation*}
x\in\Xi_{\mathit{tol}} \Ab
 \qquad\Longleftrightarrow\qquad
 \Tol (x, \mbf{A}, \mbf{b}) \geq 0 .
\end{equation*} 
\end{theorem} 
  
The tolerable solution set $\Xi_{\mathit{tol}}\Ab$ to an interval system of linear 
algebraic equations is thus a ``level set'' 
\begin{equation*} 
\bigl\{\; x\in\mathbb{R}^{n} \mid \Tol(x, \mbf{A}, \mbf{b}) \geq 0\,\bigr\} 
\end{equation*} 
of the mapping Tol with respect to the first argument $x$ under fixed $\mbf{A}$ and 
$\mbf{b}$. We will call this mapping \emph{recognizing functional} of the tolerable 
solution set, since the values of Tol are in the real line $\mbb{R}$ and their sign 
``recognizes'' the membership of a point in the set $\Xi_{\mathit{tol}}\Ab$. Below, 
we outline briefly the properties of the recognizing functional, and their detailed 
proofs can be found in \cite{SharyMCS,SharyARC,SharyBook}. 
  
  
\begin{figure}[htb]
\centering
\unitlength=1mm
\begin{picture}(100,77)
\put(0,1){\includegraphics[width=100mm]{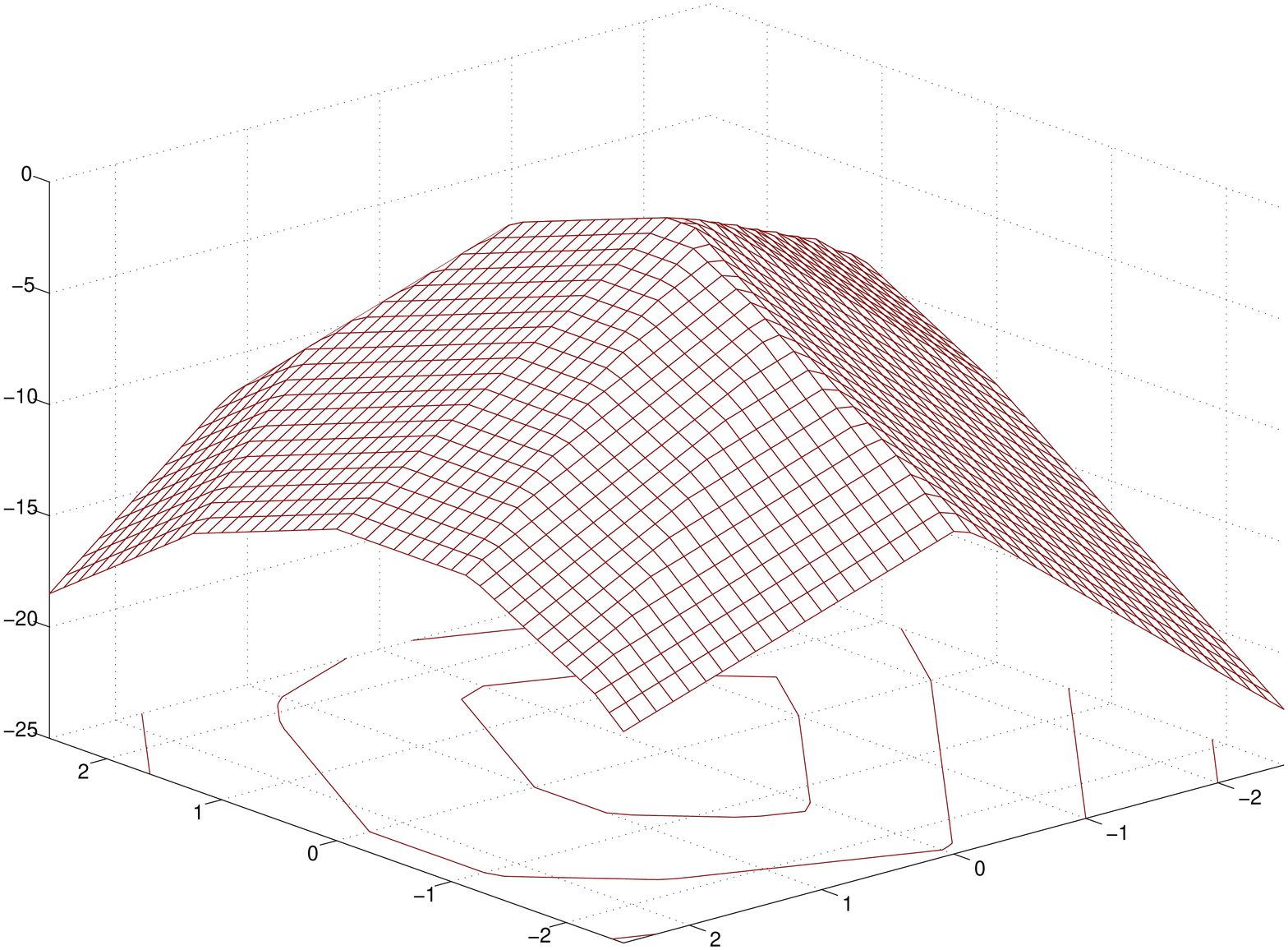}}
\put(78,5){$x_2$} 
\put(16,7){$x_1$} 
\put(-6,25){\rotatebox{90}{\small Functional values}} 
\end{picture}
\caption{Graph of the recognizing functional of the tolerable}
solution set to the interval system \eqref{ExampleInSys}. 
\label{TolGraphPic} 
\end{figure}
  
  
First of all, the functional Tol is continuous function of its arguments, which 
follows from the form of the expression that determines Tol. Moreover, Tol is 
continuous in a stronger sense, namely, it is Lipschitz continuous. At the same time, 
the functional Tol is not everywhere differentiable due to the operation ``min'' 
in its expression and ``non-smooth'' character of interval arithmetic operations. 
  
The functional $\Tol (x, \mbf{A}, \mbf{b})$ is polyhedral, that is, its hypograph 
is a polyhedral set, while its graph is composed of pieces of hyperplanes. 
  
The functional Tol is concave in the variable $x\,$ over the entire space 
${\mathbb{R}}^{n}$. Finally, the functional $\Tol (x, \mbf{A}, \mbf{b})$ attains 
a finite maximum over the whole space $\mbb{R}^{n}$. 
  
\begin{example} 
Fig.~\ref{TolGraphPic} shows the graph of the recognizing functional for the tolerable 
solution set to the interval equation system 
\begin{equation} 
\label{4x2InSystem} 
\left( 
\begin{array}{cc} 
{[-2, 0]} & {[-4,2]} \\[2pt] 
{[-3, 2]} & {[2, 3]} \\[2pt] 
{[3, 4]}  & {[4, 5]} \\[2pt] 
{[3, 5]}  & {[-2, 2]} 
\end{array}
\right) 
\,
\begin{pmatrix}
x_1 \\[2pt]  x_2 
\end{pmatrix} 
= 
\left( 
\begin{array}{c} 
{[1, 2]}  \\[2pt] 
{[-2, 0]} \\[2pt] 
{[0, 4]}  \\[2pt] 
{[-2, 3]} 
\end{array}
\right)
\end{equation} 
Polyhedral structure and nonsmoothness of the functional Tol are clearly seen 
at the picture. Also, polygons in the plane $0x_{1}x_{2}$ at Fig.~\ref{TolGraphPic} 
are level sets for various values of the level. 
\end{example} 
  
If $\Tol (x, \mbf{A}, \mbf{b}) > 0 $, then $\,x$ is a point of the topological interior 
$\mathrm{int}\;\Xi_{\mathit{tol}}\Ab$ of the tolerable solution set. It make sense 
to clarify that an interior point is a point that belongs to a set together with a ball 
(with respect to some norm) centered at this point. Therefore, interior points remain 
within the set even after small perturbations, and this fact may turn out important 
in practice. 
  
The converse is also true. Let an interval $m\times n$-system of linear algebraic 
equations $\mbf{A}x = \mbf{b}$ be such that, for every index $i = 1,2,\ldots, m$, 
there exists at least one nonzero element in the $i$-th row of the matrix $\mbf{A}$ 
or the respective right-hand side interval $\mbf{b}_i$ does not have zero endpoints. 
Then the membership $\,x\in\mathrm{int}\;\Xi_{\mathit{tol}}\Ab$ implies the strict 
inequality $\,\Tol (x, \mbf{A}, \mbf{b}) > 0$. 
  
As a consequence of the above results, we are able to perform, using the recognizing 
functional, a study of whether the tolerable solution set to an interval linear 
algebraic system is empty/nonempty. This can done according to the following procedure. 
For the interval system $\mbf{A}x = \mbf{b}\,$, we solve an unconstrained maximization 
problem for the recognizing functional $\Tol(x,\mbf{A},\mbf{b})$, that is, we compute 
$\max_{x\in\mbb{R}^n} \Tol(x,\mbf{A}, \mbf{b})$. Let $T = \max\,\Tol$, and 
it is attained at the point $\tau\in\mbb{R}^n$. Hence, 
\begin{itemize}
\item[$\bullet$]  
if $\,T\geq 0$, then $\tau\in\Xi_{\mathit{tol}}\Ab\neq\varnothing$, i.\,e., 
the tolerable solution set to the system $\mbf{A}x = \mbf{b}$ is not empty and $\tau$ 
lies inside it;
\item[$\bullet$]  
if $\,T > 0$, then $\tau\in\mathrm{int}\,\Xi_{\mathit{tol}}\Ab\neq\varnothing$, i.\,e., 
the tolerable solution set has nonempty interior and the point $\tau$ is an interior one; 
\item[$\bullet$]  
if $\;T < 0$, then $\Xi_{\mathit{tol}}\Ab = \varnothing$, i.\,e., the tolerable 
solution set to the interval equation system $\mbf{A}x = \mbf{b}$ is empty. 
\end{itemize} 
  
\begin{example} 
For the tolerable solution set to the interval linear system \eqref{4x2InSystem}, 
the graph of the recognizing functional (Fig.~\ref{TolGraphPic}) does not reach 
the zero level, all its values are negative. Hence, the tolerance problem is not 
solvable. 
  
Using the program \texttt{tolsolvty} (see below), one can compute more specific results: 
\begin{equation*} 
\max\;\Tol = -1, \qquad \mathrm{arg}\,\max\;\Tol = (-0.21294, 0)^{\top}.
\end{equation*} 
A more thorough investigation shows that, around the maximum of the functional, 
there is an entire small plateau of the constant level $-1$ (one can discern it 
in Fig.~\ref{TolGraphPic}), and the maximization method can converge to different 
points of this plateau from different initial approximations. 
\end{example} 
  
Even if the tolerable solution set is empty, the maximal value of the recognizing 
functional, $T = \max_{x\in\mbb{R}^n} \Tol(x,\mbf{A}, \mbf{b})$, can serve as a measure 
of unsolvabilty of the tolerance problem for the interval linear system. At the same time, 
the argument that delivers maximum to Tol is the ``most promising'' point with respect 
to the tolerance solvability or, in other words, it is the ``least unsolvable''. Let us 
clarify this assertion. 
  
First of all, note that widening the right-hand side vector leads to expansion of the 
tolerable solution set, i.\,e., it increases the solvability of the tolerance problem, 
while narrowing the right-hand side leads to reduction of the tolerable solution set, 
i.\,e., decreases the solvability of the interval tolerance problem. Consequently, 
the value of the coordinated contraction of the right-hand sides to the point at which 
the tolerable solution set becomes empty can be taken as a solvability measure for the 
interval linear tolerance problem. Conversely, the minimal value of the coordinated 
expansion of the intervals in the right-hand sides, under which the tolerable solution 
set becomes nonempty, characterizes an ``unsolvability measure'' of the tolerance problem. 
Similar natural considerations are widely used in interval data fitting (see, e.\,g., 
\cite{JangBaker,Zhilin}). The ``coordinated'' expansion or narrowing of the data intervals 
is usually understood as uniform expansion or contraction relative to their centers. 
  
The point (or points) that first appears in non-empty tolerable solution set during 
the uniform expansion of the right-hand side intervals is of special interest to us, 
since it delivers the smallest ``incompatibility'' to the interval tolerance problem. 
So, this point (or points) can be taken as a pseudo-solution of the original equation 
system. 
  
The remarkable fact is that the argument of $\max\Tol$ is the first point that appears 
in the non-empty tolerable solution set after uniform, with respect to its midpoint, 
widening of the right-hand side vector. To substantiate it, let us look at the expression 
for the recognizing functional Tol: 
\begin{equation*}
\Tol (x,\mbf{A},\mbf{b})\; = \,\min_{1\leq i\leq m} 
  \left\{\, \r\mbf{b}_i - \left|\;\m\mbf{b}_i - \sum_{j = 1}^n \,\mbf{a}_{ij} x_{j} 
  \,\right|\,\right\}. 
\end{equation*} 
The quantities $\r\mbf{b}_i$ enter as addons in all subexpressions over which we take 
$\min_{1\leq i\leq m}$ when calculating the final value of the functional. Therefore, 
if we denote 
\begin{equation*} 
\mbf{e} = \bigl([-1, 1], \ldots, [-1, 1] \bigr)^{\top}, 
\end{equation*} 
i.\,e., the interval vector with the radii of all components equal to $1$, then 
the system $\mbf{A}x = \mbf{b} + C\mbf{e}\,$ has the widened right-hand sides and 
their radii become $\r\mbf{b}_{i} + C$, $i = 1,2,\ldots,m$. We thus have 
\begin{equation*}
\Tol (x,\mbf{A},\mbf{b} + C\mbf{e}) = \Tol (x,\mbf{A},\mbf{b}) + C. 
\end{equation*} 
Consequently, 
\begin{equation*}
\max_{x}\;\Tol (x,\mbf{A},\mbf{b} + C\mbf{e}) = \max_{x}\;\Tol(x,\mbf{A},\mbf{b}) + C, 
\end{equation*} 
which proves our assertion. 
  
We can see that the values of the recognizing functional at a point give a quantitative 
measure of the compatibility of this point with respect to the tolerable solution set 
of a given interval linear system. Consequently, the argument of the maximum of the 
recognizing functional, no matter whether it belongs to a nonempty tolerable solution 
set or not, corresponds to the maximum tolerance compatibility for a given interval 
linear system. That is why we regard it as a pseudo-solution of the original system 
of linear algebraic equations, to which interval regularization is applied. 
  
Next, we consider the interesting question of what result will be produced by interval 
regularization for the case when the matrix of the system and its right-hand side vector 
are specified exactly, without errors and uncertainty. 
  
If the matrix $\mbf{A}$ of the linear system and its right-hand side vector $\mbf{b}$ are 
point (non-interval), i.\,e. $\mbf{A} = A = (a_{ij})$ and $\mbf{b} = b = (b_{i})$, then 
\begin{equation*}
\r\mbf{b}_{i} = 0, \hspace{18mm} \m\mbf{b}_{i} = b_{i}, \hspace{18mm} \mbf{a}_{ij} = a_{ij} 
\end{equation*} 
for all $i, j$. The recognizing functional of the solution set then takes the form 
\begin{align*}
\Tol(x, A,b) 
&= \min_{1\leq i\leq m}
  \left\{\, -\biggl|\,b_{i} - \sum_{j=1}^n a_{ij}x_j \,\biggr|\,\right\} \ 
 = \  -\max_{1\leq i\leq m} \  \biggl|\,b_{i} - \sum_{j=1}^n a_{ij}x_j \,\biggr| \\[3mm] 
&= \ -\max_{1\leq i\leq m} \;\left|\,\bigl(Ax)_{i} - b_{i}\,\right|              \\[3mm] 
&= \ -\left\| Ax - b\,\right\|_{\infty}. 
\end{align*} 
Through $\|\cdot\|_\infty$, we denote Chebyshev norm ($\infty$-norm) of a vector 
in the finite-dimensional space $\mbb{R}^m$, which is defined as $\|y\|_{\infty} = 
\max_{1\leq i\leq m}\,|y_{i}|$. Then 
\begin{equation*}
\max\;\Tol(x)\, = \,\max_{x\in\mbb{R}^n}\;\bigl(-\| Ax - b\,\|_{\infty}\bigr)\, 
  = -\min_{x\in\mbb{R}^n}\; \| Ax - b\,\|_{\infty}, 
\end{equation*} 
insofar as $\,\max\,(-f(x)) = -\min\, f(x)$. In this particular case, the maximization 
of the recognizing functional is equivalent, therefore, to minimizing the Chebyshev norm 
of defect of the solution, very popular in data processing. 
  
In practice, the maximization of the recognizing functional can be performed with 
the use of nonsmooth optimization methods that have been greatly developed in the 
last decades. The author used for this purpose the so-called $r$-algorithms, invented 
by Naum Shor \cite{ShorZhurb} and later elaborated in V.M.\,Glushkov Institute of Cybernetics 
of the National Academy of Sciences of Ukraine \cite{StetsyukBook,StetsyukJCT}. Based on 
the computer code \texttt{ralgb5} created by Petro Stetsyuk, a free program \texttt{tolsolvty} 
has been written for Scilab and \textsc{Matlab},  available at \cite{IntervalAnalysis}. 
Our computational experience shows that \texttt{tolsolvty} works satisfactorily for 
the linear systems having the condition number which is not large. One more possibility 
of implementation of the approach can be based on the separating planes algorithms 
of non-smooth optimization, proposed in \cite{Nurminski,VorontsovaLNCS,VorontsovaJCT, 
VorontsovaIEEE}. 
  
To summarize, in the interval regularization method for the system of linear algebraic 
equations \eqref{LinSystem1}--\eqref{LinSystem2}, the matrix $A$ ``inflates'' by a small 
value to result in an interval matrix $\mbf{A}$. In particular, if the equation system 
is determined imprecisely, then the intervalization of $A$ to $\mbf{A}$ can be carried 
out based on the information of the accuracy to which the elements of $A$ and $b$ are 
given. We thus get an interval system of linear algebraic equations $\mbf{A}x = \mbf{b}$ 
with $\mbf{A}\ni A$ and $\mbf{b}\ni b$. Then we compute numerically unconstrained 
maximum, with respect to $x$, of the recognizing functional $\Tol(x,\mbf{A},\mbf{b})$ 
of the tolerable solution set for the interval linear system $\mbf{A}x = \mbf{b}$. 
The argument of the maximum value of Tol is the sought-for pseudosolution 
to the equation system \eqref{LinSystem1}--\eqref{LinSystem2}.

  
\section{Formal (algebraic) approach}

Yet another way of estimating the tolerable solution set to interval systems 
of equations is the \emph{formal approach} (sometimes called \emph{algebraic}). 
It consists in replacing the initial estimation problem with the problem of 
computing the so-called formal (algebraic) solution for a special interval 
equation or a system of equations. Based on the formal approach, we can propose 
one more version of the interval regularization procedure. 
  
\begin{definition}
An interval (interval vector, matrix) is called a \emph{formal solution} 
to the interval equation (system of equations, inequalities, etc.) if substituting 
this interval (interval vector, matrix) into the equation (system of equations, 
inequalities, etc.) and executing all interval arithmetic, analytic, etc., operations 
result in a valid relation. 
\end{definition}
  
The formal solutions correspond, therefore, to the usual general mathematical concept 
of a solution to an equation. Introduction of a special term for them in connection 
with interval equations has, rather, historical causes. Formal solutions turn out 
to be very useful in estimating various solution sets for interval systems of equations 
(see, e.\,g., \cite{SainzGardeJorba,ModalInAnal,SharyOneMore,SharySurvey,SharyBook}). 
The simplest result of this kind applies to the tolerable solution set and looks 
as follows: 
  
\begin{theorem} 
\label{InnEstTheo}
If an interval vector $\mbf{x}\in\mbb{IR}^n$ is a formal solution to the interval 
linear system $\mbf{A}x = \mbf{b}$, then $\mbf{x}\subseteq\Xi_{\mathit{tol}}\Ab$, 
that is, $\mbf{x}$ is an inner interval estimate of the tolerable solution set 
$\Xi_{\mathit{tol}}\Ab$. 
\end{theorem} 
  
\begin{proof} 
Let us recall that the point $\tilde{x}\in\mbb{R}^n$ lies in the tolerable solution 
set $\Xi_{\mathit{tol}}\Ab$ of an interval system of linear algebraic equations 
$\mbf{A}x = \mbf{b}$, if and only if $\,\mbf{A}\cdot\tilde{x}\,\subseteq\,\mbf{b}$. 
  
If the interval vector $\mbf{x}\in\mbb{IR}^n$ is a formal solution to the interval 
linear system $\mbf{A}x = \mbf{b}$, then, for any point $x\in\mbf{x}$, 
\begin{equation*} 
\mbf{A}x \;\subseteq\; \mbf{A}\mbf{x} \; = \;\mbf{b} 
\end{equation*} 
due to inclusion monotonicity. Hence, we can assert the membership 
$x\in\Xi_{\mathit{tol}}\Ab$ for every such $x\in\mbf{x}$, which implies 
$\mbf{x}\subseteq\Xi_{\mathit{tol}}\Ab$. 
\end{proof} 
  
It is worth noting that the result of Theorem~\ref{InnEstTheo} is, in fact, a particular 
case of a~very general results on inner estimation of the so-called AE-solution sets 
for interval systems of equations \cite{ModalInAnal,SharySurvey,SharyBook}. A remarkable 
property of the formal approach to the inner estimation of the solution sets to interval 
linear systems is that it produces interval estimates maximal with respect 
to inclusion \cite{IreneMaxEst,SharySurvey,SharyBook}. 
  
\begin{example} 
The formal solution to the interval linear system of equations \eqref{ExampleInSys} 
is the interval vector 
\begin{equation} 
\label{InnerEst} 
\begin{pmatrix} 
[-0.147059, 0.147059] \\[2pt] 
[-0.147059, 0.147059] \\[2pt] 
[-0.147059, 0.147059] 
\end{pmatrix}. 
\end{equation} 
We can see that it really provides an inner box within the tolerable solution 
set for the system \eqref{ExampleInSys}.\footnote{The formal solutions may be 
computed, e.\,g., using the code \texttt{subdiff} available 
at \url{http://www.nsc.ru/interval/shary/Codes/progr.html}.} It is even inclusion 
maximal in the sense that there do not exist interval boxes being inner estimates 
of the tolerable solution set and including \eqref{InnerEst} as a proper subset 
at the same time. 
\end{example} 
  
Computation of formal solutions for interval linear systems of equations is well 
developed in modern interval analysis. Over the past decades, several numerical 
methods have been designed that can efficiently compute formal solutions. These 
are various stationary single-step iterations \cite{Kupriyanova,Markov,ModalInAnal, 
SharyBook} and the subdifferential Newton method \cite{SharyOneMore,SharyBook}. 
Most of these methods work in the so-called \emph{Kaucher complete interval 
arithmetic} (see, e.\,g., \cite{Kaucher,SharySurvey,SharyBook}) which consists 
of usual ``proper'' intervals $[\un{x}, \ov{x}]$ with $\un{x}\leq\ov{x}$ as well 
as ``improper'' intervals $[\un{x}, \ov{x}]$ with $\un{x} > \ov{x}$. Kaucher 
interval arithmetic has better algebraic properties than the classical interval 
arithmetic and, in addtition, it allows to work adequately with interval 
uncertainties of various types \cite{SharySurvey,SharyBook}. 
  
\begin{example} 
\label{ForSolExmp2}
The interval equation $[1, 2]\,x = [3, 4]$ does not have proper formal solutions, 
while its formal solution in Kaucher interval arithmetic is improper interval 
$[3, 2]$. It cannot be interpreted as an inner interval estimate of the tolerable 
solution set according to Theorem~\ref{InnEstTheo}. The situation is explained 
by the fact that the tolerable solution set is empty in this case. 
\end{example} 
  
Let us turn to the interval regularization for a system of linear algebraic 
equations $Ax = b$. We intervalize it and thus get an interval linear system 
$\mbf{A}x = \mbf{b}$. Next, we compute its formal solution $\mbf{x}^\ast$. 
As a pseudo-solution of the original linear system $Ax = b$, we can take 
the middle of the vector $\mbf{x}^\ast$, that is, the point $x^{\ast}\in\mbb{R}^n$ 
with the coordinates 
\begin{equation} 
\label{FormPseSol} 
x_i^\ast = \m\mbf{x}_i^{\ast} \  \stackrel{\text{def}}{=} \ 
  \tfrac{1}{2} \bigl(\un{\mbf{x}}^\ast_{i} + \ov{\mbf{x}}^\ast_{i}\bigr), 
  \qquad i = 1,2,\ldots,n. 
\end{equation} 
If the formal solution $\mbf{x}^\ast$ is proper, then the motivation for such a choice 
of the pseudo-solution is clear. In view of Theorem~\ref{InnEstTheo} and further results, 
$\mbf{x}^\ast$ is the maximal, with respect to inclusion, inner interval box within 
the tolerable solution set. Therefore the middle point of $\mbf{x}^\ast$ is really one 
of the ``most representative'' points from the tolerable solution set. But if the formal 
solution $\mbf{x}^\ast$ is improper (as in Example~\ref{ForSolExmp2}), then the choice 
of $x^\ast$ in the form of \eqref{FormPseSol} requires explanation. 
  
If the formal solution of the intervalized system of equations $\mbf{A}x = \mbf{b}$ 
is improper, then its tolerable solution set is most likely empty. But the choice 
of a pseudo-solution in the form of \eqref{FormPseSol} ensures an ``almost minimal'' 
measure of the ``tolerance unsolvability'' of this point in the sense that it requires 
the smallest widening of the right-hand side $\mbf{b}$ to obtain a non-empty tolerable 
solution set. 
  
We recall that, for interval linear systems of the form $Ax = \mbf{b}$ with a point 
square matrix $A$, a unique formal solution exists if and only if the matrix satisfies 
the \emph{absolute regularity property} \cite{Markov,SharyOneMore,SharySurvey,SharyBook}. 
Among several equivalent formulation of this property, the simplest one is that both 
$A$ and the matrix $|A|$ (composed of the modules of the elements) should be nonsingular 
\cite{Markov,SharySurvey,SharyBook}.\footnote{This property was also called ``complete 
regularity'' and ``$\iota$-regularity'' in earlier works.} 
  
If the point matrix $A$ is absolutely regular and $\mbf{b}$ is a proper interval 
vector, then it is easy to substantiate that the formal solution to the linear system 
$Ax = \mbf{b}$ is also proper. In addition, the point \eqref{FormPseSol}, i.\,e., 
the midpoint of the inner interval box for $\Xi_{\mathit{tol}}\Ab$ provides the maximum 
value of the recognizing functional Tol. From continuity reasons, it follows that 
the same holds true for sufficiently narrow interval matrices $\mbf{A}\ni A$ too. 
Hence, the formal solutions to such interval linear systems are proper, they can be 
interpreted as inner interval boxes for the corresponding tolerable solution sets, 
and the instruction \eqref{FormPseSol} makes good sense. 
  
For slightly wider, but still sufficiently narrow interval matrices for which 
the interval linear systems has formal solutions that are not entirely proper, the same 
continuity reasons imply that the recipe \eqref{FormPseSol} gives us points which are not 
far from the optimal point arg max Tol. 
  
It is worthwhile to note that, inflating the right-hand side vector, we can always make 
the point \eqref{FormPseSol} fall into a non-empty tolerable solution set. Indeed, let 
$\mbf{x}^\ast$ be a formal solution to the interval equation system $\mbf{A}x = \mbf{b}$ 
and $\mbf{e} = ([-1, 1], \ldots, [-1, 1])^{\top}$ is the $n$-vector of all $[-1, 1]$'s. 
If $\mbf{x}^\ast$ is improper in some components, then we take the vector $\mbf{x}^{\ast} 
+ t\mbf{e}$. It satisfies $\m\!(\mbf{x}^{\ast} + t\mbf{e}) = \m\mbf{x}^\ast$, and all 
the components of $\mbf{x}^{\ast} + t\mbf{e}\,$ become proper for $t \geq t^{\ast}$, where 
\begin{equation*} 
t^{\ast} \  \stackrel{\mathrm{def}}{=} \  
\tfrac{1}{2} \Bigl|\,\min_{i}\;(\ov{\mbf{x}}^\ast_{i} - \un{\mbf{x}}^\ast_{i})\Bigr|. 
\end{equation*} 
Also, the point \eqref{FormPseSol} belongs to $\mbf{x}^{\ast} + t\mbf{e}$, i.\,e. 
\begin{equation} 
\label{MidPointIncl} 
\m\mbf{x}^\ast \in \;\mbf{x}^{\ast} + t\mbf{e} 
\end{equation} 
for such $t\geq t^{\ast}$. We can assert that then 
\begin{equation} 
\label{InflEstim} 
\mbf{A}(\m\mbf{x}^{\ast}) \  
\subseteq \  \mbf{A}(\mbf{x}^{\ast} + t\mbf{e}) \  
\subseteq \  \mbf{A}\mbf{x}^{\ast} + \mbf{A}(t\mbf{e}) \ 
  = \  \mbf{b}^{\ast} + t\mbf{Ae}
\end{equation} 
due to \eqref{MidPointIncl}, inclusion monotonicity of the interval arithmetic operations 
and subdistributivity of multiplicaition with respect to addition for proper $\mbf{A}$. 
As a consequence, the point $\m\mbf{x}^\ast$ lies in the non-empty tolerable solution set 
of the interval linear system 
\begin{equation*} 
\mbf{A}x\, = \,\mbf{b} + t^{\ast}\mbf{Ae} 
\end{equation*} 
with the uniformly widened right-hand side vector. In fact, the exact equality 
$\mbf{A}(\mbf{x}^{\ast} + t\mbf{e}) = \mbf{A}\mbf{x}^{\ast} + \mbf{A}(t\mbf{e})$ holds 
true instead of inclusion, since both $t\mbf{e}$ and $\mbf{A}(t\mbf{e})$ are balanced 
intervals (symmetric with respect to zero). This makes our estimate in \eqref{InflEstim} 
even sharper. 
  
  
\section*{Conclusion}

The work proposes a new approach to regularization of ill-conditioned and imprecise 
systems of linear algebraic equations based on interval analysis methods, and we call 
it \emph{interval regularization}. Its essence is the ``immersion'' of the original 
system of equations into an interval system of the same structure for which the 
so-called tolerable solution set is studied, the most stable of the solution sets. 
As a pseudo-solution of the original system of equations, we assign a point from 
the tolerable solution set (if it is not empty) or a point providing the largest 
``tolerable'' compatibility (if this solution set is empty). To find such a point, 
one can apply numerical methods for computing formal (algebraic) solutions of 
interval systems or, alternatively, algorithms of non-smooth optimization for 
computing the maximum of the recognizing functional of the tolerarble solution set. 
  
The interval regularization has two strengths. First, for the system of linear 
equations $Ax = b$, it depends on the properties of the matrix $A$ significantly 
less than in other approaches. The properties of $A$ are taken into account as 
if automatically, by the method itself. Second, information about the data uncertainty, 
both in the matrix $A$ and right-hand side vector $b$, is taken into account very 
simply and naturally. One only need to further ``inflate'' the interval matrix and/or 
the right-hand side vector according to the known accuaracy level. 
  
An interesting open question is the choice of the extent to which we should 
``inflate'' the matrix $A$ of the original system \eqref{LinSystem1}--\eqref{LinSystem2}. 
The wider the interval matrix $\mbf{A}$ of the system of equations \eqref{InLSystem1}%
--\eqref{InLSystem2} or \eqref{InLSystem3}, the more well-conditioned matrices are in it, 
the more stable the tolerable solution set according to \eqref{SetTheoRepr} and, hence, 
the better the general regularization of the problem. On the other hand, for a wider 
interval matrix $\mbf{A}$, much different from the original point matrix $A$, 
the solution of the regularized problem can be strongly distorted in comparison with 
the solution of the original system. Consequently, how should we choose optimally 
the widths of the elements of the interval matrix $\mbf{A}\ni A$? If the original 
system \eqref{LinSystem1}--\eqref{LinSystem2} is specified imprecisely and inaccurately, 
as an interval equation system \eqref{InLSystem1}--\eqref{InLSystem2} with a predetermined 
accuracy level, then the question is solved naturally. In the general case, an additional 
study is necessary. 
  
A certain drawback of the new approach stems from the fact that, in order to construct 
the desired pseudo-solution, we may have to process some of (or even many) ``endpoint'' 
linear systems of the regularized interval equation system, and some of these endpoint 
systems can have worse conditionality than the original system. However, the reality 
of this danger depends on the way the interval regularization method is implemented. 
Therefore, technological issues related to the implementation of the corresponding 
numerical methods are very important, but their development is beyond the scope 
of our article.


  
\end{document}